\documentclass[12pt]{amsart}
\usepackage{amssymb,amscd}
\usepackage{verbatim}

\usepackage{amsmath,amssymb,graphicx,mathrsfs}   
\usepackage[colorlinks=true,allcolors = blue]{hyperref} 

\let\frak\mathfrak

\def\>{\relax\ifmmode\mskip.666667\thinmuskip\relax\else\kern.111111em\fi}
\def\<{\relax\ifmmode\mskip-.333333\thinmuskip\relax\else\kern-.0555556em\fi}
\def\vsk#1>{\vskip#1\baselineskip}
\def\vv#1>{\vadjust{\vsk#1>}\ignorespaces}
\def\vvn#1>{\vadjust{\nobreak\vsk#1>\nobreak}\ignorespaces}

  \let\ssize\scriptstyle
\let\sssize\scriptscriptstyle

\let\Medskip\medskip
\def\medskip{\par\Medskip}
\let\Bigskip\bigskip
\def\bigskip{\par\Bigskip}

\let\Maketitle\maketitle
\def\maketitle{\Maketitle\thispagestyle{empty}\let\maketitle\empty}

\newtheorem{thm}{Theorem}[section]
\newtheorem{cor}[thm]{Corollary}
\newtheorem{lem}[thm]{Lemma}

\theoremstyle{definition}                                  

\numberwithin{equation}{section}

\theoremstyle{definition}

\let\mc\mathcal
\let\nc\newcommand

\let\ka\kappa
\let\la\lambda

\let\phi\varphi

\let\Om\Omega

\let\der\partial

\let\ox\otimes

\let\geq\geqslant

\let\leq\leqslant

\let\on\operatorname
\let\bi\bibitem
\let\bs\boldsymbol

\def\C{{\mathbb C}}
\def\Z{{\mathbb Z}}

\def\F{{\mathbb F}}

\def\+#1{^{\{#1\}}}

\def\beq{\begin{equation}}
\def\eeq{\end{equation}}
\def\be{\begin{equation*}}
\def\ee{\end{equation*}}

\nc{\bea}{\begin{eqnarray*}}
\nc{\eea}{\end{eqnarray*}}
\nc{\bean}{\begin{eqnarray}}
\nc{\eean}{\end{eqnarray}}
\nc{\Ref}[1]{{\rm(\ref{#1})}}

\def\g{{\mathfrak g}}

\let\ga\gamma
\let\Ga\Gamma

\nc{\Il}{{\mc I_{\bs\la}}}
\nc{\bla}{{\bs\la}}
\nc{\Fla}{\F_\bla}
\nc{\tfl}{{T^*\Fla}}
\nc{\GL}{{GL_n(\C)}}
\nc{\GLC}{{GL_n(\C)\times\C^*}}

\let\sd s 

\def\ddk_#1{\kk_{#1}\<\>\frac\der{\der\<\>\kk_{#1}}}

\def\bul{\mathbin{\raise.2ex\hbox{$\sssize\bullet$}}}
\def\intt{\mathchoice
{\mathop{\raise.2ex\rlap{$\,\,\ssize\backslash$}{\intop}}\nolimits}
{\mathop{\raise.3ex\rlap{$\,\sssize\backslash$}{\intop}}\nolimits}
{\mathop{\raise.1ex\rlap{$\sssize\>\backslash$}{\intop}}\nolimits}
{\mathop{\rlap{$\sssize\<\>\backslash$}{\intop}}\nolimits}}

\let\kk q 
\let\cc c

\let\Ko K

\def\GZ/{Gelfand-Zetlin}
\def\KZ/{{\slshape KZ\/}}
\def\qKZ/{{\slshape qKZ\/}}
\def\XXX/{{\slshape XXX\/}}

\nc{\A}{{\mc C}}

\def\FFF{{\mathbb F}}

\def\sll{{\frak{sl}}}

\begin{document}

\hrule width0pt
\vsk->

\hrule width0pt
\vsk->

\title[Hyperelliptic integrals modulo $p$ and Cartier-Manin matrices]
{Hyperelliptic integrals modulo $p$\\  and Cartier-Manin matrices}

\author[Alexander Varchenko]
{ Alexander Varchenko}

\maketitle

\begin{center}
{\it Department of Mathematics, University
of North Carolina at Chapel Hill\\ Chapel Hill, NC 27599-3250, USA\/}

\vsk.5>
{\it Faculty of Mathematics and Mechanics, Lomonosov Moscow State
University\\ Leninskiye Gory 1, 119991 Moscow GSP-1, Russia\/}
\end{center}

{\let\thefootnote\relax
\footnotetext{\vsk-.8>
\noindent
{\sl E\>-mail}:\enspace anv@email.unc.edu\>,
supported in part by NSF grant DMS-1665239}}

\vsk>
{\leftskip3pc \rightskip\leftskip \parindent0pt \Small
{\it Key words\/}: KZ equations, hyperelliptic integrals, Cartier-Manin matrix, reduction to \\ 
\phantom{aaaaaaaaaa} characteristic $p$
\vsk.6>
{\it 2010 Mathematics Subject Classification\/}: 13A35 (33C60, 32G20)
\par}

\begin{abstract}

The hypergeometric solutions of the KZ equations were constructed almost 30 years ago. The polynomial
solutions of the KZ equations over the finite field $\F_p$  with a prime number $p$ of elements were constructed 
recently.  In this paper we consider the example of the KZ equations whose hypergeometric solutions are given by
hyperelliptic integrals of genus $g$. It is known that in this case the 
total $2g$-dimensional space of holomorphic
solutions is given by
the hyperelliptic integrals. We show that the recent construction of the polynomial solutions over the field $\F_p$
in this case gives only a $g$-dimensional space of solutions, that is,
a  "half" of what the complex analytic construction gives.
We also show that all the constructed polynomial solutions over the field $\F_p$ can be obtained by reduction modulo $p$ of a single
distinguished hypergeometric solution. The corresponding formulas involve the entries
of the Cartier-Manin matrix of the hyperelliptic curve.

That situation is analogous to the example of the elliptic integral considered in the classical Y.I.\,Manin's paper in 1961.

\end{abstract}

{\small\tableofcontents\par}

\setcounter{footnote}{0}
\renewcommand{\thefootnote}{\arabic{footnote}}

\section{Introduction}

The hypergeometric solutions of the KZ equations were constructed almost 30 years ago,
see \cite{SV1, SV2}. The polynomial
solutions of the KZ equations over the finite field $\F_p$  
with a prime number $p$ of elements were constructed 
recently in \cite{SV3}.  In this paper we consider the example of the KZ equations whose hypergeometric solutions are given by
hyperelliptic integrals of genus $g$. It is known that in this case the total $2g$-dimensional space of
holomorphic solutions is given by
the hyperelliptic integrals. We show that the recent construction of the polynomial solutions over the field $\F_p$
in this case gives only a $g$-dimensional space of solutions, that is,
a  "half" of what the complex analytic construction gives.
We also show that all the constructed polynomial solutions over the field $\F_p$ can be obtained by reduction modulo $p$ of a single
distinguished hypergeometric solution. The corresponding formulas involve the entries
of the Cartier-Manin matrix of the hyperelliptic curve.

That situation is analogous to the example of the elliptic integral considered in the classical Y.I.\,Manin's paper in 1961.

\smallskip
The paper is organized as follows. In Section \ref{sec KZ} we describe the KZ equations, and construct
for them  two types of solutions: over $\C$ and over $\F_p$.
In Section \ref{sec I^m(z)} we show that the solutions, constructed over $\F_p$, form a module,
denoted by $\mc M_{g,p}$,  of rank $g$. 
In Section \ref{sec binom} useful formulas on binomial coefficients are collected.
In Section \ref{sec J^m(z)} a new basis of the module $\mc M_{g,p}$ is constructed.
In Section \ref{sec CM} the Cartier-Manin matrix of a hyperelliptic curve is defined.
In Section \ref{sec comp} we introduce a distinguished holomorphic solution of the KZ equations,
reduce its Taylor expansion coefficients modulo $p$ and express this reduction in terms of 
the polynomial solutions over $\F_p$ and entries of the Cartier-Manin matrix.

\smallskip
The author thanks R.\,Arnold, F.\,Beukers, N.\,Katz, V.\,Schechtman, J.\,Stienstra, Y. Zarhin, and 
W.\,Zudilin
 for useful discussions. The author thanks MPI in Bonn for hospitality in
May-June 2018 when this work had been finished.

\section{KZ equations}
\label{sec KZ}

\subsection{Description of equations}   Let $\g$ be a simple Lie algebra  over the field $\C$,
$\Om \in\g^{\ox 2} $ the Casimir element corresponding to an invariant scalar product on $\g$, \
$V_{1}, \dots, V_{n}$ finite-dimensional irreducible $\g$-modules.

The system of KZ equations with parameter $\ka\in \C^\times$
on a $\ox_{i=1}^nV_{i}$-valued function 
$I(z_1,\dots,z_n)$ is the system of the differential equations 
\bean
\label{KZg}
 \frac{\der I}{\der z_i} = 
\frac 1\ka \sum_{j\ne i}\frac{\Om^{(i,j)}}{z_i-z_j} I,\qquad i=1,\dots,n,
\eean
where $\Om^{(i,j)}$ is the Casimir element acting in the $i$-th and $j$-th factors,
see \cite{KZ, EFK}.  The KZ differential equations commute with the action of $\g$ on 
$\ox_{i=1}^nV_{i}$, in particular, they preserve the subspaces of singular vectors
of given weight.

In \cite{SV1, SV2} the KZ equations restricted to the subspace of singular vectors of given weight were identified with 
a suitable Gauss-Manin differential equations and the corresponding solutions of the KZ equations were
presented as multidimensional hypergeometric integrals. 

\smallskip

Let $p$ be a prime number and $\F_p$ the field with $p$ elements. Let $\g^p$ be the same Lie algebra considered over 
 $\F_p$. 
Let $V_{1}^p, \dots, V_{n}^p$ be the $\g^p$-modules which are reductions modulo $p$ of
$V_{1}, \dots, V_{n}$, respectively.
If $\ka$ is an integer and $p$ large enough with respect to $\ka$, then one can look for solutions
$I(z_1,\dots,z_n)$ of the KZ equations in $\ox_{i=1}^nV_{i}^p\ox \F_p[z_1,\dots,z_n]$.
Such solutions were constructed in \cite{SV3}.

In this paper we address two questions: 
\begin{enumerate}
\item[A.]
What is the number of independent solutions constructed in \cite{SV3} for given $\F_p$?
\item[B.]
How are those solutions related to the solutions over $\C$,  that 
are given by hypergeometric integrals?
\end{enumerate}
We answer these question in the example in which the hypergeometric solutions are presented by hyperelliptic
integrals. 

\medskip
The object of our study is the following systems of 
equations. For a positive integer $g$ and
 $z=(z_1,\dots, z_{2g+1})\in \C^{2g+1}$, we study the
 column vectors  $I(z)=(I_1(z)$, \dots, $I_{2g+1}(z))$ satisfying  the  system of
 differential and algebraic linear equations:
\bean
\label{KZ}
\phantom{aaa}
 \frac{\partial I}{\partial z_i} \ = \
   {\frac 12} \sum_{j \ne i}
   \frac{\Omega^{(i,j)}}{z_i - z_j}  I ,
\quad i = 1, \dots , {2g+1},
\qquad
I_1(z)+\dots+I_{2g+1}(z)=0,
\eean
where
\[ \Omega^{(i,j)} \ = \ \begin{pmatrix}
             & \vdots^i &  & \vdots^j &  \\
        {\scriptstyle i} \cdots & {-1 } & \cdots &
                 1 & \cdots \\
                   & \vdots &  & \vdots &   \\
        {\scriptstyle j} \cdots & 1 & \cdots & {-1 }&
                 \cdots \\
                   & \vdots &  & \vdots &
                   \end{pmatrix} ,
                    \]
and all other  entries equal  zero.

The system of equations \Ref{KZ} is the system of the KZ differential equations with parameter $\ka=2$ associated with the Lie algebra $\sll_2$ and the subspace of singular vectors of weight $2g-1$ of the tensor power 
$(\C^2)^{\ox {(2g+1)}}$ of two-dimensional irreducible $\sll_2$-modules, up to a gauge transformation, see 
this example in  \cite[Section 1.1]{V2}.

\subsection{Solutions of \Ref{KZ} over $\C$} 
Consider the {\it master function}
\bean
\label{mast f}
\Phi(t,z_1,\dots,z_{2g+1}) = \prod_{a=1}^{2g+1}(t-z_a)^{-1/2}
\eean
and  the ${2g+1}$-vector  of hyperelliptic integrals
\bean
\label{Iga}
I^{(\ga)} (z)=(I_1(z),\dots,I_{2g+1}(z)),
\eean
 where
\bean
\label{s}
I_j=\int  \Phi(t,z_1,\dots,z_{2g+1}) \frac {dt}{t-z_j},\qquad j=1,\dots,{2g+1}.
\eean
The integrals are over an element $\ga$ of the first homology group
 $\ga$ of the hyperelliptic curve with equation
\bea
y^2 = (t-z_1)\dots (t-z_{2g+1}).
\eea
Starting from such $\ga$,  chosen for given $\{z_1,\dots,z_{2g+1}\}$, the vector $I^{(\ga)}(z)$ can be analytically continued as a multivalued holomorphic function of $z$ to the complement in $\C^n$ to the union of the
diagonal  hyperplanes $z_i=z_j$.

\begin{thm}
\label{thm1.1}

 The vector $I^{(\ga)}(z)$ satisfies  the KZ equations  \Ref{KZ}.
\end{thm}

Theorem \ref{thm1.1} is a classical statement probably known in the 19th century. 
Much more general algebraic and differential equations satisfied by analogous multidimensional hypergeometric integrals were considered in \cite{SV1, SV2}.  Theorem \ref{thm1.1} is discussed as an example in  \cite[Section 1.1]{V2}.

\begin{thm} [{\cite[Formula (1.3)]{V1}}]
\label{thm dim}

All solutions of the KZ equations \Ref{KZ} have this form. 
Namely, the complex vector space of solutions of the form \Ref{Iga} is $2g$-dimensional.

\end{thm}

This theorem follows from the determinant formula for multidimensional hypergeometric integrals  in    \cite{V1}, in particular,
from \cite[Formula (1.3)]{V1}.

\subsection{Solutions of KZ equations \Ref{KZ} over $\F_p$} 
We  always assume that the prime number $p$ satisfies the inequality
\bean
p\geq 2g+1.
\eean
Define the {\it master polynomial}
\bean
\label{mast p}
\Phi_p(t,z_1,\dots,z_{2g+1}) = \prod_{a=1}^{2g+1}(t-z_a)^{(p-1)/2}  \ \ \in \ \F_p[t,z]
\eean
and  the ${2g+1}$-vector  of polynomials
\bean
\label{F v}
P(z)=(P_1(t,z),\dots,P_{2g+1}(t,z)),\qquad P_j(t,z)=  \frac {1}{t-z_j}\Phi_p(t,z_1,\dots,z_{2g+1}).
\eean
Consider the Taylor expansion
\bean
\label{te}
P(t,z)= {\sum}_{i=0}^{(p-1)/2 + gp -g-1}  P^i(z) t^i, \qquad P^i(z)=(P^i_1(z),\dots,P^i_{2g+1}(z)),
\eean
with $P^i_j(z) \in\F_p[z]$.

\begin{thm}
[\cite{SV3}]
\label{thm SV}
For every positive integer $l$, the vector $P^{lp-1}(z)$ satisfies the KZ equations \Ref{KZ}.

\end{thm}

This statement is a particular case of \cite[Theorem 2.4]{SV3}. Cf. Theorem \ref{thm SV} with \cite{K}.

Theorem \ref{thm SV} gives exactly $g$ solutions $P^{p-1}(z),\dots, P^{gp-1}(z)$. We denote
\bea
I^m(z)= (I^m_1(z),\dots,I^m_{2g+1}(z)), 
\eea
where
\bean
\label{P I}
I^m(z) := P^{(g-m)p-1}(z), \qquad m=0,\dots,g-1.
\eean

\section{Linear independence of solutions $I^m(z)$}
\label{sec I^m(z)}

Denote $\F_p[z^p]:=\F_p[z_1^p,\dots,z_{2g+1}^p]$. 
The set of all solutions $I(z)\in \F_p[z]^{2g+1}$ of the KZ equations
\Ref{KZ} 
is a module over the ring
$\F_p[z^p]$ since equations \Ref{KZ} are linear and
$\frac{\der z_i^p}{\der z_j} =0$ in $\F_p[z]$ for all $i,j$.
Denote by
\bea
\mc M_{g,p}=\big\{ \sum_{m=0}^{g-1} c_m(z) I^m(z) \ |\ c_m(z)\in\F_p[z^p]\big\},
\eea
the $\F_p[z^p]$-module generated by $I^m(z)$, $m=0,\dots,g-1$.

\begin{thm}
\label{thm inde}
Let $p\geq 2g+1$. The solutions $I^{m}(z)$,  $m=0,\dots, g-1$,
are linear independent over the ring $\F_p[z^p]$,
that is, if \,$\sum_{m=0}^{g-1}   c_m(z)I^m(z)=0$ for some $c_m(z)\in\F_p[z^p]$, then
$c_m(z)=0$ for all $m$.

\end{thm}

\begin{proof}
For $m=0,\dots,g-1$,  the coordinates of the vector $I^m(z)$ are homogeneous polynomials in $z$ of degree
$(p-1)/2 + mp -g$ and
\bea
I^{m}_j(z)
=  {\sum}I_{j;\ell_1,\dots,\ell_{2g+1}}^mz_1^{\ell_1}\dots z_{2g+1}^{\ell_{2g+1}},
\eea
where the sum is over the elements of the set
\bea
\Ga^{m}_j = \{ (\ell_1,\dots,\ell_{2g+1})\in \Z^{2g+1}_{\geq 0}
\ |\
\sum_{i=1}^{2g+1} \ell_j=(p-1)/2 +mp-g, 
\\
\phantom{aaaaaaaaaaaaaa}
 0\leq \ell_j\leq (p-3)/2, \ \ 0\leq \ell_i\leq (p-1)/2\ \on{for}\ i\ne j\}
\eea
and
\bea
I_{j;\ell_1,\dots,\ell_{2g+1}}^m = (-1)^{(p-1)/2+mp-g}
\binom{(p-3)/2}{\ell_j}\prod_{i\ne j}\binom{(p-1)/2}{\ell_i} \ \ \in\ \ \FFF_p.
\eea
Notice that all coefficients $I_{j;\ell_1,\dots,\ell_{2g+1}}^m $ are nonzero.
Hence each solution $I^{m}(z) $ is nonzero.

We show that already the first coordinates  $I^m_1(z)$, $m=0,\dots,g-1$, 
are linear independent over the ring $\F_p[z]$.

Let $\bar \Ga^{m}_1 \subset \FFF_p^{2g+1}$ be the image of the
set $\Ga^{m}_1$ 
under the natural projection $\Z^{2g+1}\to \FFF_p^{2g+1}$. 
The points of $\bar \Ga^{m}_1$ are in bijective correspondence with the points of
$\Ga^{m}_1$. Any two sets $\bar \Ga^{m}_1$ and $\bar \Ga^{m'}_1$ do not intersect,
if $m\ne m'$.  (The sets $\bar \Ga^{m}_1$ are analogs in $\F_p^{2g+1}$ of the
Newton polytopes of the polynomials
$I^m_1(z)$.)

For any $m$ and any nonzero polynomial $c_m(z)\in \FFF_p[z_1^p,\dots, z_{2g+1}^p]$,
consider
the nonzero polynomial $c_m(z) I^{m}_1(z)\in \FFF_p[z_1,\dots,z_{2g+1}]$ and the set
$\Ga^{m}_{1,c_m}$ of points $\ell\in\Z^{2g+1}$ such that the monomial
$z_1^{\ell_1}\dots z_{2g+1}^{\ell_{2g+1}}$ enters $c_m(z) I^m_1(z)$ with nonzero coefficient.
Then the natural projection of $\Ga^{m}_{1,c_m}$ to $\FFF_p^{2g+1}$
coincides with $\bar \Ga^{m}_1$.
Hence the polynomials
$I^{m}_1(z)$,  $m=0,\dots, g-1$,
are linear independent over the ring $\F_p[z^p]$.
\end{proof}

\section{Binomial coefficients modulo $p$} 
\label{sec binom}

In this section we collect useful formulas on binomial coefficients.

\subsection{Lucas's theorem}

\begin{thm} [\cite{L}]
\label{thm L}
For non-negative integers $m$ and $n$ and a prime $p$,
 the following congruence relation holds:
\bean
\label{LT}
\binom{m}{n}\equiv \prod _{i=0}^k \binom{m_i}{n_i}\quad (\on{mod}\ p),
\eean
where $m=m_{k}p^{k}+m_{k-1}p^{k-1}+\cdots +m_{1}p+m_{0}$ and $n=n_{k}p^{k}+n_{k-1}p^{k-1}+\cdots +n_{1}p+n_{0}$
are the base $p$ expansions of $m$ and $n$ respectively. This uses the convention that 
$\binom{m}{n}=0$ if $m<n$.
\qed
\end{thm}

\begin{lem}
\label{lem 2a} For $a\in\Z_{>0}$, we have
\bea
\binom{2a}{a} \not\equiv 0\quad (\on{mod}\,p)
\eea
if and only if the base $p$ expansion of  $a=a_0+a_1p + a_2 p^2+\dots+a_kp^k$
has the property:
\bea
a_i \leq\frac{p-1}2\quad\on{for}\quad  i=0,\dots,k.
\eea
In that case
\bean
\label{2a}
\binom{2a}{a}\equiv \prod _{i=0}^k \binom{2a_i}{a_i}\quad (\on{mod}\ p).
\eean

\end{lem}

The lemma is a corollary of Lucas's theorem.

\subsection{Useful identities}  For $0\leq k\leq (p-3)/2$, we have
\bean
\label{binom{(p-3)/2}{k}}
\binom{(p-3)/2}{k} 
&=& 
\binom{(p-1)/2}{k}\frac{(p-3)/2 -k+1}{(p-1)/2}
=
\binom{(p-1)/2}{k}\frac{p-2k-1}{p-1}
\\
\notag
&\equiv&
\binom{(p-1)/2}{k}(2k+1)
\quad
(\text{mod}\ p),
\eean
for $0\leq k\leq (p-1)/2$ 
\bean
\label{binom{(p-3)/2}{k-1}}
\binom{(p-3)/2}{k-1} 
&=& 
\binom{(p-1)/2}{k}\frac{k}{(p-1)/2}
=
\binom{(p-1)/2}{k}\frac{2k}{p-1}
\\
\notag
&\equiv&
\binom{(p-1)/2}{k}(-2k)
\quad
(\text{mod}\ p).
\eean
For a positive integer $k$, 
\bean
\label{-1/2k}
{-1/2 \choose k} &=& \frac {(-1/2)(-1/2 -1)\cdot  \cdot  \cdot (-1/2 - (k-2))(-1/2-(k-1))} {k!} 
\\
\notag
&=&
(-2)^{-k} \frac {1 \cdot 3 \cdot 5 \cdot ... \cdot (2k-1)}{k!}
=(-1)^k 2^{-k} \frac {(2k)!/(2 \cdot 4 \cdot 6 \cdot 8 \cdot ... \cdot 2k)}{k!}
\\
\notag
&
=&
(-1)^k 2^{-k} \frac {(2k)!/(2^k k!)}{k!}
=(-4)^k  \binom{2k}{k},
\eean
for $0\leq k\leq (p-1)/2$
\bean
\label{-1/2 k}
\binom{(p-1)/2}{k}\equiv (-4)^{-k}\binom{2k}k \qquad (\on{mod}\, p).
\eean

\section{Solutions $J^m(z)$}
\label{sec J^m(z)}

\subsection{Sets $\Delta^r_s$}
We introduce sets that are used later. 
For $r=0,\dots,g-1$, $s=0,\dots,g$,  define 
\bean
\label{Del mj}
\phantom{aaaaaa}
\Delta^{r}_s 
=
 \{(\ell_3,\dots,\ell_{2g+1})\in \Z^{2g-1}_{\geq 0}
\ | \  0\leq \sum_{i=3}^{2g+1} \ell_i +s-rp\leq (p-1)/2,\ \ 
 \ell_i \leq (p-1)/ 2 \}.
\eean

\subsection{Definition} Introduce the vectors $J^m(z)\in \F_p[z]^{2g+1}$, $m=0,\dots,g-1$, by the formula
\bean
\label{JI}
J^m(z)=\sum_{l=0}^m  I^{m-l}(z)z_1^{lp}  \binom{g-m -1+ l}{g-m-1},
\eean
that is, 
\bea
J^0(z) &=& I^0(z),
\\
J^1(z) &=&  I^{0}(z)z_1^{p} \binom{g-1}{g-2}  + I^1(z),
\\
J^2(z) &=&  I^{0}(z)z_1^{2p} \binom{g-1}{g-3}  +  I^{1}(z)z_1^{p} \binom{g-2}{g-3}  + I^2(z), 
\eea
and so on.

\begin{lem}
\label{lem ns} For $m=0,\dots,g-1$, the vector $J^m(z)$ is a solution of the KZ equations \Ref{KZ}.
Moreover, the $\F_p[z^p]$-module spanned by $J^m(z)$, $m=0,\dots,g-1$, coincides with the 
$\F_p[z^p]$-module $\mc M_{g,p}$ spanned by $I^m(z)$, $m=0,\dots,g-1$.
\qed
\end{lem}

For the vector $P(t,z)$ in \Ref{te}, consider the Taylor expansion
\bean
\label{te+z}
P(t+z_1,z)= {\sum}_{i=0}^{(p-1)/2 + gp -g-1}  \tilde
P^i(z) t^i.
\eean

\begin{lem}
\label{lem PJ} For $m=0,\dots,g-1$, we have
\bean
\label{PJ}
J^m(z) = \tilde P^{(g-m)p-1}(z),
\eean
cf. formula \Ref{P I}.

\end{lem}

\begin{proof} We have
$P(t,z)= {\sum}_{i=0}^{(p-1)/2 + gp -g-1}  P^i(z) t^i$, hence
\bea
P(t+z_1,z)= {\sum}_{i=0}^{(p-1)/2 + gp -g-1}  P^i(z) (t+z_1)^i = {\sum}_{i=0}^{(p-1)/2 + gp -g-1}  P^i(z) 
\sum_{j=0}^i \binom{i}{j}t^j z_1^{i-j}.
\eea
If $p\not| (i+1)$, then $\binom{i}{(m-g)p-1} \equiv 0$ (mod $p$) by Lucas's theorem.  Hence
\bea
&&
\tilde P^{(g-m)p-1}(z) =P^{(g-m)p-1}(z) \binom{(g-m)p-1}{(g-m)p-1}+
P^{(g-m+1)p-1}(z)z_1^p \binom{(g-m+1)p-1}{(g-m)p-1}
\\
&&
\phantom{aaa}
+P^{(g-m+2)p-1}(z)z_1^{2p} \binom{(g-m+2)p-1}{(g-m)p-1} + \dots
\\
&&
\phantom{aaa}
= I^m(z) + I^{m-1}(z) z_1^p  \binom{g-m}{g-m-1} + I^{m-2}(z)z_1^{2p}  \binom{g-m+1}{g-m-1} + \dots,
\eea
where the last equality holds also by Lucas's theorem. This gives the lemma.
\end{proof}

\subsection{Formula for $J^m(z)$}
Denote
\bean
\label{laj}
\la_j=\frac{z_j-z_1}{z_2-z_1},
\qquad j=1,\dots,2g+1.
\eean

\begin{thm}
\label{thm J}
For $m=0,\dots,g-1$, we have
\bean
\label{CFoo}
J^m(z)= (z_2-z_1)^{(p-1)/2 +mp-g} K^m(\la),
\eean
where
\bean
\label{JK(la)}
K^m(\la)={\sum}_{\ell\in\Delta^m_g} K^m_\ell(\la),
\eean
 $\Delta^m_j$ is defined in \Ref{Del mj}, and
\bean
\label{Jme}
&&
K^m_\ell(\la)= (-1)^{(p-1)/2 +mp-g} 
\binom{(p-1)/2}{\sum_{i=3}^{2g+1}\ell_i +g-mp} 
{\prod}^{2g+1}_{i=2} 
\binom{(p-1)/2}{\ell_{i}} \la_s^{\ell_3}\dots\la_{2g+1}^{\ell_{2g+1}}
\\
\notag
&&
\phantom{aaaaaaa}
\times
(1, -2\sum_{i=3}^{2g+1}\ell_i -2g, 2\ell_3+1, 
\dots, 2\ell_{2g+1}+1).
\eean
\end{thm}

Using \Ref{-1/2 k} we may rewrite formula  \Ref{Jme} as
\bean
\label{Jmee}
&&
K^m_\ell(\la)= (-1)^{(p-1)/2}4^{-2\sum_{i=3}^{2g+1}\ell_i -g+mp}
\\
\notag
&&
\phantom{aaa}
\times\binom{2\sum_{i=3}^{2g+1}\ell_i +2g-2mp}{\sum_{i=3}^{2g+1}\ell_i+g -mp} 
{\prod}^{2g+1}_{i=2} 
\binom{2\ell_i}{\ell_{i}} \la_s^{\ell_3}\dots\la_{2g+1}^{\ell_{2g+1}}
\\
\notag
&&
\phantom{aaaaaaa}
\times
(1, -2\sum_{i=3}^{2g+1}\ell_i -2g, 2\ell_3+1, 
\dots, 2\ell_{2g+1}+1).
\eean

\begin{proof} We have
\bea
&&
P((z_2-z_1)x+z_1,z) = (z_2-z_1)^{(p-1)/2 + gp -g-1}
\\
&&
\phantom{aaa}
\times
x^{(p-1)/2}(x-1)^{(p-1)/2}\prod_{j=3}^{2g+1}(x-\la_j)^{(p-1)/2}\Big(\frac1x, \frac 1{x-1},
\frac 1{x-\la_3},\dots,\frac 1{x-\la_{2g+1}}\Big)
\eea
and
\bea
P((z_2-z_1)x+z_1,z) = {\sum}_{i=0}^{(p-1)/2 + gp -g-1}  \tilde
P^i(z) (z_2-z_1)^ix^i.
\eea
Hence $J^m(z)=\tilde P^{(g-m)p-1}(z)$ equals the coefficient of $x^{(g-m)p-1}$ in 
\bean
\label{mp in v}
x^{(p-1)/2}(x-1)^{(p-1)/2}\prod_{j=3}^{2g+1}(x-\la_j)^{(p-1)/2}\big(\frac1x, \frac 1{x-1},
\frac 1{x-\la_3},\dots,\frac 1{x-\la_{2g+1}}\big)
\eean
multiplied by $(z_2-z_1)^{(p-1)/2 + mp -g}$. 
 We have
\bea
&&
(z_2-z_1)^{-(p-1)/2 - mp +g}J_1^m(z) =
\\
&&
\phantom{aaaa}
(-1)^{(p-1)/2 + mp -g}{\sum} \binom{(p-1)/2}{\ell_2}
\dots \binom{(p-1)/2}{\ell_{2g+1}} \la_3^{\ell_3}\dots\la_{2g+1}^{\ell_{2g+1}},
\eea
where the sum is over the set
\bea
\Delta 
&=&
 \{(\ell_2,\dots,\ell_{2g+1})\in \Z^{2g}_{\geq 0}
\ | \  \sum_{i=2}^{2g+1} \ell_i = mp-g+(p-1)/2,\ \ 
\\
&&
\phantom{aaaa}
\ell_j \leq (p-1)/ 2, \  j=2,\dots,2g+1\}.
\eea
Expressing $\ell_2$ from the conditions defining $ \Delta $ we  write 
\bea
&&
(z_2-z_1)^{-(p-1)/2 - mp +g}J_1^m(z) 
=
(-1)^{(p-1)/2 + mp -g}
\\
&&
\phantom{aaa}
\times{\sum}
\binom{(p-1)/2}{\sum_{i=3}^{2g+1}\ell_i +g- mp} \binom{(p-1)/2}{\ell_3}
\dots \binom{(p-1)/2}{\ell_{2g+1}} \la_3^{\ell_3}\dots\la_{2g+1}^{\ell_{2g+1}},
\eea
where the sum is over the set
\bea
\Delta^m_g
&=&
 \{(\ell_3,\dots,\ell_{2g+1})\in \Z^{2g-1}_{\geq 0}
\ | \  0\leq \sum_{i=3}^{2g+1} \ell_i +g-mp\leq (p-1)/2,\ \ 
\\
&&
\phantom{aaaa}
 \ell_i \leq (p-1)/ 2, \  i=3,\dots,2g+1\}.
\eea
Similarly we have
\bea
&&
(z_2-z_1)^{-(p-1)/2 - mp +g}J_2^m(z) =
(-1)^{(p-1)/2+mp-g}
\\
&&
\phantom{aaa}
\times
{\sum} \binom{(p-3)/2}{\ell_2} \binom{(p-1)/2}{\ell_{3}}
\dots \binom{(p-1)/2}{\ell_{2g+1}} \la_3^{\ell_3}\dots\la_{2g+1}^{\ell_{2g+1}},
\eea
where the sum is over the set
\bea
\Delta'
&=&
 \{(\ell_2,\dots,\ell_{2g+1})\in \Z^{2g}_{\geq 0}
\ | \  \sum_{i=2}^{2g+1} \ell_i =  mp-g+(p-1)/2,\ \ 
\\
&&
\phantom{aaaa}
\ell_2 \leq (p-3)/ 2 \ \text{and}\ 
\ell_i \leq (p-1)/ 2 \ \on{for}\  i> 2\}.
\eea
Expressing $\ell_2$ from the conditions defining $\Delta'$ we write 
\bea
&&
(z_2-z_1)^{-(p-1)/2 - mp +g}J_2^m(z)
=
(-1)^{(p-1)/2+mp-g}
\\
&&
\phantom{aaa}
\times
{\sum}
\binom{(p-3)/2}{\sum_{i=3}^{2g+1}\ell_i +g- mp-1} \binom{(p-1)/2}{\ell_3}
\dots \binom{(p-1)/2}{\ell_{2g+1}} \la_3^{\ell_3}\dots\la_{2g+1}^{\ell_{2g+1}},
\eea
where the sum is over the set
\bea
 \Delta''
&=&
 \{(\ell_3,\dots,\ell_{2g+1})\in \Z^{2g-1}_{\geq 0}
\ | \  0\leq\sum_{i=3}^{2g+1} \ell_i +g-mp-1\leq (p-3)/2,\ \ 
\\
&&
\phantom{aaaa}
\ell_j \leq (p-1)/ 2, \  j=3,\dots,2g+1\}.
\eea
For $j=3,\dots,2g+1$, we have
\bea
&&
(z_2-z_1)^{-(p-1)/2 - mp +g}J_j^m(z)
=
(-1)^{(p-1)/2+mp-g}
\\
&&
\phantom{aaa}
\times
{\sum}\binom{(p-3)/2}{\ell_j}{\prod}^{2g+1}_{i=2,\ i\ne j} 
\binom{(p-1)/2}{\ell_{i}}
 \la_3^{\ell_3}\dots\la_{2g+1}^{\ell_{2g+1}},
\eea
where the sum is over the set
\bea
 \Delta'''
&=&
 \{(\ell_2,\dots,\ell_{2g+1})\in \Z^{2g}_{\geq 0}
\ | \  \sum_{i=2}^{2g+1} \ell_i =  mp-g+(p-1)/2,\ \ 
\\
&&
\phantom{aaaa}
 \ell_j \leq (p-3)/ 2 \ \text{and}\ 
 \ell_i \leq (p-1)/ 2, \  i\ne j\}.
\eea
Expressing $\ell_2$ from the conditions defining $\Delta'''$ we write 
\bea
&&
(z_2-z_1)^{-(p-1)/2 - mp +g}J_j^m(z)
=
(-1)^{(p-1)/2+mp-g}
\\
&&
\phantom{aaaa}
\times
{\sum}
\binom{(p-1)/2}{\sum_{i=3}^{2g+1}\ell_i+g - mp} 
\binom{(p-3)/2}{\ell_j}{\prod}^{2g+1}_{i=2,\ i\ne j} 
\binom{(p-1)/2}{\ell_{i}} \la_3^{\ell_3}\dots\la_{2g+1}^{\ell_{2g+1}},
\eea
where the sum is over the set
\bea
\bar \Delta''''
&=&
 \{(\ell_3,\dots,\ell_{2g+1})\in \Z^{2g-1}_{\geq 0}
\ | \  0\leq \sum_{i=3}^{2g+1} \ell_i +g-mp\leq (p-1)/2,\ \ 
\\
&&
\phantom{aaaa}
\ell_j \leq (p-3)/ 2 \ \text{and}\ 
 \ell_i \leq (p-1)/ 2, \  i\ne j\}.
\eea
Using identities  \Ref{binom{(p-3)/2}{k}}, \Ref{binom{(p-3)/2}{k-1}}
we may rewrite  $J_j^m(z)$, $j=2,\dots,2g+1$, in the form indicated in the theorem.
\end{proof}

\section{Cartier-Manin matrix}
\label{sec CM}
Consider the hyperelliptic curve $X$ with equation
\bea
y^2=x(x-1)(x-\la_3)\dots(x-\la_{2g+1}),
\eea
where $\la_3,\dots,\la_{2g+1}\in\F_p$, while, in the previous section,
$\la_3,\dots,\la_{2g+1}$ were rational functions in $z$, see fromula \Ref{laj}.

Following \cite{AH} define the $g\times g$ {\it Cartier-Manin matrix} $C(\la)=(C^r_s(\la))_{s,r=0}^{g-1}$ of that
curve. Namely, for $s=0,\dots,g-1$, expand
\bea
x^{g-s-1} \big(x (x-1) (x-\la_3)\dots (t-\la_{2g+1})\big)^{(p-1)/2}
={\sum}_k Q^k_s x^k
\eea
with $Q^k_j\in\F_p$ and   set
\bean
\label{CmjQ}
C^r_s(\la) := Q^{(g-r)p-1}_s, \qquad  r=0,\dots,g-1.
\eean
The Cartier-Manin matrix represents
the action of the Cartier operator on the space of holomorphic differentials of the hyperelliptic curve. 
That operator is dual to the Frobenius operator on the cohomology group 
$H^1(X,\mc O_X)$,
see for example, \cite{AH}.

\begin{lem}
\label{lem CM}

We have
\bean
\label{CIJ}
C^r_s(\la)
&=& {\sum}_{\ell\in\Delta^r_s} C^r_{s;\,\ell}(\la),
\eean
where  $\Delta^r_s$ is defined in \Ref{Del mj} and
\bean
\label{Cij}
&&
C^r_{s;\,\ell}(\la)=(-1)^{(p-1)/2 +rp -s}
\binom{(p-1)/2}{\sum_{i=3}^{2g+1}\ell_i+s-rp}
 \prod_{i=3}^{2g+1}\binom{(p-1)/2}{\ell_i} 
\la_3^{\ell_3}\dots\la_{2g+1}^{\ell_{2g+1}}.
\eean
\qed
\end{lem}

The lemma is proved by straightforward calculation similar to the proof of Theorem \ref{thm J}.

We may rewrite \Ref{Cij} as
\bean
\label{Cij 1}
&&
C^r_{s;\,\ell}(\la)
=
(-1)^{(p-1)/2}  4^{-2\sum_{i=3}^{2g+1}\ell_i-s+rp}
\\
\notag
&&
\phantom{aaaaaaaaa}
\times
\binom{2\sum_{i=3}^{2g+1}\ell_i+2s-2rp}{\sum_{i=3}^{2g+1}\ell_i+s-rp}
 \prod_{i=3}^{2g+1}\binom{2\ell_i}{\ell_i} 
\la_3^{\ell_3}\dots\la_{2g+1}^{\ell_{2g+1}}.
\eean

\section{Comparison of solutions over $\C$ and $\F_p$}
\label{sec comp}

Now we will
\begin{enumerate}
\item
 distinguish one holomorphic solution of the KZ equations, 
 \item
  expand it into the Taylor series, 
  \item
 for any $p\geq 2g+1$ reduce this Taylor expansion modulo $p$, 
  \item
  observe in that reduction of the Taylor expansion
   all polynomial solutions, that we have constructed and nothing more.

\end{enumerate}

\subsection{Distinguished holomorphic solution}  
Recall that holomorphic solutions of our KZ equations have the form
 $ I(z)=(I_1(z),\dots,I_{2g+1}(z))$,  where
\bea
I_j(z)
=
\int_{\ga} \frac{dt}{\sqrt{(t-z_1)\dots (t-z_{2g+1})}}\,\frac 1{t-z_j}
\eea
and $\ga$ is an oriented curve on the hyperelliptic curve with equation
 $y^2= (t-z_1)\dots(t-z_{2g+1})$. Assume that $z_3,\dots,z_{2g+1}$ are  closer to
 $z_1$ than to $z_2$:
 \bea
 \Big| \frac{z_j-z_1}{z_2-z_1}\Big|<\frac 12, \qquad
 j=3,\dots,2g+1.
 \eea
Choose $\ga$ to be the circle 
 $  \big\{ t\in \C\ |\ \big|\frac{t-z_1}{z_2-z_1}\big|=\frac12\big\}$
 oriented counter-clockwise, and multiply the vector $I(z)$ by the normalization constant
 $1/2\pi$.
 
 We call this solution $I(z)$ the {\it distinguished} solution.

\subsection{Rescaling} 
Change variables and write
\bean
\label{IK}
I(z_1,\dots,z_{2g+1})=(z_2-z_1)^{-1/2-g} L(\la_3,\dots,\la_{2g+1}),
\eean
where 
\bea
(\la_3,\dots,\la_{2g+1}) =\Big(\frac{z_3-z_1}{z_2-z_1},\dots,\frac{z_{2g+1}-z_1}{z_2-z_1}\Big),&
\eea
$L(\la)
=
(L_1,\dots,L_{2g+1})$,
\bea
L_j
= \frac{-1}{2\pi}
\int _{|x|=1/2}
\frac{dx}{\sqrt{x(x-1)(x-\la_3)\dots (x-\la_{2g+1})}}\,\frac 1{x-\la_j},
\eea
and we set 
$\frac 1{x-\la_1}:=\frac1x$,
$\frac 1{x-\la_2}:=\frac1{x-1}$.

 The function $L(\la)$ is holomorphic at the point $\la=0$.  Hence
 \bea
 \label{Ls}
L(\la)=
\sum_{(k_3,\dots,k_{2g+1})\in \Z^{2g-1}_{\geq 0}}L_{k_3,\dots,k_{2g+1}}\la_3^{k_3}\dots\la_{2g+1}^{k_{2g+1}},
\eea
where the coefficients lie in $\Z[\frac12]^{2g+1}$. Hence for any $p\geq 2g+1$, this power
series can be projected to a formal power series in $\F_p[\la]^{2g+1}$.

We relate this power series and the polynomial solutions $J^m(z)$, $m=0,\dots,g-1$, constructed earlier.

\subsection{Taylor expansion of $L(\la)$}

\begin{lem}
\label{lem K(0)}
We have
\bean
\label{K(0)}
L(0,\dots,0)=(-1)^g \binom{-1/2}{g}(1,-2g,1,\dots,1).
\eean
\end{lem}

\begin{proof} We have $ \frac{-1}{2\pi}=\frac {(-1)^{-1/2}}{2\pi i}$ and 
\bea
L_1(0,\dots,0)
&=&
 \frac {(-1)^{-1/2}}{2\pi i}\int _{|x|=1/2} (x-1)^{-1/2}\frac{dx}{x^{g+1}}
=
\frac {1}{2\pi i}\int _{|x|=1/2} (1-x)^{-1/2}\frac{dx}{x^{g+1}}
\\
&=&
\frac {1}{2\pi i}\int _{|x|=1/2} \sum_{k=0}^\infty(-1)^kx^k\binom{-1/2}k
\frac{dx}{x^{g+1}}
=(-1)^g\binom{-1/2}g,
\eea
\bea
&&
L_2(0,\dots,0)
=
 \frac {(-1)^{-1/2}}{2\pi i}\int _{|x|=1/2} (x-1)^{-3/2}\frac{dx}{x^{g}}
=
-\frac {1}{2\pi i}\int _{|x|=1/2} (1-x)^{-3/2}\frac{dx}{x^{g}}
\\
&&
\phantom{a}
=
-\frac {1}{2\pi i}\int _{|x|=1/2} \sum_{k=0}^\infty(-1)^kx^k\binom{-3/2}k\frac{dx}{x^{g}}
=(-1)^{g}\binom{-3/2}{g-1}
=(-1)^{g}\binom{-1/2}{g}(-2g).
\eea
The coordinates $L_j(0,\dots,0)$ for $j>2$ are calculated similarly.
\end{proof}

\begin{lem}
\label{lem TaC}
We have
\bean
\label{ex5}
L(\la)=
\sum_{(k_3,\dots,k_{2g+1})\in \Z^{2g-1}_{\geq 0}}L_{k_3,\dots,k_{2g+1}}\la_3^{k_3}\dots\la_{2g+1}^{k_{2g+1}},
\eean
where
\bean
\label{CFo}
L_{k_3,\dots, k_{2g+1}} \!\!
&=&
(-1)^{g}
\binom{-1/2}{k_3+\dots+k_{2g+1}+g}
\prod_{i=3}^{2g+1}\binom{-1/2}{k_i} 
\\
\notag
&&
\times
(1, -2k_3-\dots-2k_{2g+1}-2g, 2k_3+1,\dots,2k_{2g+1}+1).
\eean
\end{lem}
\begin{proof}
The  proof is similar to the proof of Lemma \ref{lem K(0)}.
\end{proof}

Using formula \Ref{-1/2k}
we may reformulate \Ref{CFo} as 
\bean
\label{C2m}
&&
 L_{k_3,\dots,k_{2g+1}} = 4^{-2(k_3+\dots+k_{2g+1})-g} 
\\
\notag
&&
\phantom{aasaaaass}
\times
\binom{2(k_3+\dots+k_{2g+1}+g)}{k_3+\dots+k_{2g+1}+g} 
\binom{2k_3}{k_3}\dots
\binom{2k_{2g+1}}{k_{2g+1}} 
\\
\notag
&&
\phantom{aaaaaasaaass}
\times 
(1, -2k_3-\dots-2k_{2g+1}-2g, 2k_3+1,\dots,
2k_{2g+1}+1). 
\eean

\subsection{Coefficients,  nonzero modulo $p$}
Given $(k_3,\dots,k_{2g+1})\in\Z^{2g-1}_{\geq 0}$, let 
\bea
k_i = k_i^0 + k_i^1p +\dots+ k^a_ip^{a}, \quad 0\leq k_i^j\leq p-1, \quad i=3,\dots,2g+1,
\eea
be the $p$-ary expansions. Assume that $a$ is such that not all numbers
$k^a_i$, $i=3,\dots,2g+1$, are equal to zero. By Lemma \ref{lem 2a}, the product
$\prod_{i=3}^{2g+1}\binom{2k_i}{k_i}$ is not congruent to zero modulo $p$ if and only if
\bean
\label{admk}
k_i^j\leq \frac{p-1}2 \qquad\on{ for\  all}\ \  i,\ j.
\eean
Assume that condition \Ref{admk} holds. Then for any $j=0,\dots,a,$ we have
\bea
\sum_{i=3}^{2g+1}k_i^j 
\leq (2g-1) \frac{p-1}2=gp -g - \frac{p-1}2 < gp.
\eea
Define the {\it shift coefficients} $(m_0, \dots, m_{a+1})$ as follows. Namely, put
$m_0=g$. We have $\sum_{i=3}^{2g+1}k^0_i+g<gp$. Hence there exists a unique integer
$m_1$, $0\leq m_1<g$, such that
\bea
0\leq \sum_{i=3}^{2g+1}k^0_i+g-m_1p<p.
\eea
We have $\sum_{i=3}^{2g+1}k^1_i+m_1 < gp$. Hence there exists a unique integer $m_2$, 
$0\leq m_2<g$, such that
\bea
0\leq \sum_{i=3}^{2g+1}k^1_i+m_1-m_2p<p,
\eea
and so on. We have $0\leq m_j<g$ for all $j=1,\dots, a+1$.

We say that a tuple $(k_3,\dots,k_{2g+1})$ is {\it admissible} if it has property \Ref{admk} and
its shift coefficients
$(m_0,\dots,m_{a+1})$  satisfy the system of inequalities
\bean
\label{shc}
\sum_{i=3}^{2g+1}k_i^j - m_{j+1}p + m_j \leq \frac{p-1}2,
\quad 
j=0,\dots,a.
\eean

\begin{thm}
\label{thm nzc}

We have $L_{k_3,\dots, k_{2g+1}}\not\equiv 0$ (mod $p$)  if and only if
the tuple $(k_3,\dots,k_{2g+1})$ is admissible. The tuple $(k_3,\dots,k_{2g+1})$ is admissible, 
if and only if 
$(k^j_3,\dots,k^j_{2g+1})\in \Delta^{m_{j+1}}_{m_j}$ for $j=0,\dots,a$, 
where the sets $\Delta^r_s$ are defined in \Ref{Del mj}. If the tuple  $(k_3,\dots,k_{2g+1})$
is admissible, then modulo $p$ we have
\bean
\label{Lkk}
&&
L_{k_3,\dots, k_{2g+1}}\la_3^{k_3}\dots\la_{2g+1}^{k_{2g+1}}
\equiv (-1)^{a(p-1)/2}\binom{2m_{a+1}}{m_{a+1}}
\\
\notag
&&
\phantom{aaa}
\times
\Big(
\prod_{j=1}^a C^{m_{j+1}}_{m_j; k^j_3,\dots,k^j_{2g+1}}(\la_3^{p^j}, \dots,\la_{2g+1}^{p^j})
\Big)
K^{m_{1}}_{k^0_3,\dots,k^0_{2g+1}}(\la_3, \dots,\la_{2g+1}),
\eean
 where $C^r_{s;\,\ell}(\la)$ are terms of the Cartier-Manin matrix expansion in \Ref{CIJ}
 and $K^m_\ell(\la)$ are the terms of the expansion in \Ref{CFoo} of the solution $J^m(z)$.

\end{thm}

\begin{proof} We have $L_{k_3,\dots, k_{2g+1}}\not\equiv 0$ (mod $p$) if and only if each of the binomial coefficients
in \Ref{C2m} is not divisible by $p$. For all $i=3, \dots,2g+1$, we have $\binom{2k_i}{k_i}\not\equiv 0$ (mod $p$) 
if and only if property \Ref{admk} holds.

The  $p$-ary expansion of $k_3+\dots+k_{2g+1}+g$ is
\bea
&&
k_3+\dots+k_{2g+1}+g = \Big(\sum_{i=3}^{2g+1}k_i^0 - m_1p + g\Big)
\\
&&
\phantom{aaaa}
+\Big(\sum_{i=3}^{2g+1}k_i^1 - m_2p + m_1\Big)p
+\dots
+\Big(\sum_{i=3}^{2g+1}k_i^a - m_{a+1}p + m_a\Big)p^a  + m_{a+1}p^{a+1}.
\eea
By Lemma \ref{2a}, the binomial coefficient 
$\binom{2(k_3+\dots+k_{2g+1}+g)}{k_3+\dots+k_{2g+1}+g} $ 
is not divisible by $p$ if and only if  inequalities \Ref{shc} hold. 
Thus   $L_{k_3,\dots, k_{2g+1}}\not\equiv 0$ (mod $p$)  if and only if
the tuple $(k_3,\dots,k_{2g+1})$ is admissible.

The statement that the tuple $(k_3,\dots,k_{2g+1})$ is admissible, 
if and only if 
$(k^j_3,\dots,k^j_{2g+1})\in \Delta^{m_{j+1}}_{m_j}$ for $j=0,\dots,a$, 
follows from  the definition of the sets $\Delta^r_s$.

The last statement of the theorem is a straightforward corollary
of Lucas's theorem, formulas for $C^r_{s;\ell}(\la)$, $K^m_\ell(\la)$, 
and the fact that
$4^{kp}\equiv 4^k$ (mod $p$) for any $k$.
\end{proof}

\subsection{Decomposition of $L(\la)$ into the disjoint sum of polynomials} Define a set
\bean
\label{set S}
\phantom{aaaaaa}
M=\{(m_0,\dots,m_{a+1})\, |\, a\in \Z_{\geq 0}, \, m_0=g, \, m_j\in\Z_{\geq 0},\, 
m_j<g\ \on{for}\ j=1,\dots, a+1\}.
\eean 
For any $\vec m=(m_0,\dots,m_{a+1})\in M$, define  the $2g+1$-vector
of polynomial in $\la=(\la_3$, \dots, $\la_{2g+1})$:
\bean
\label{Kvec}
&&
K_{\vec m}(\la) = (-1)^{a(p-1)/2}\binom{2m_{a+1}}{m_{a+1}}
\\
\notag
&&
\phantom{aaa}
\times
\Big(\prod_{j=1}^a C^{m_{j+1}}_{m_j}(\la_3^{p^j}, \dots,\la_{2g+1}^{p^j})
\Big)
K^{m_{1}}(\la_3, \dots,\la_{2g+1}).
\eean
Notice that for $\vec m, \vec m'\in M$, $\vec m\ne\vec m'$, the set of monomials, entering
with nonzero coefficients the polynomial
$K_{\vec m}(\la)$, does not intersect the set of monomials, entering
with nonzero coefficients the polynomial $K_{\vec m'}(\la)$.

\begin{cor}
\label{cor main}
We have
\bean
\label{KK vec}
L(\la) \equiv {\sum}_{\vec m\in M}K_{\vec m}(\la) \quad (\on{mod}\,p).
\eean
\end{cor}

Notice that  by Lemma \ref{lem TaC}, $L(\la)$ is a power series in $\la$ with coefficients
in $\Z^{2g+1}\big[\frac12\big]$ independent of $p$, while the right-hand side
in \Ref{KK vec} is a formal infinite sum of polynomials in $\la$ with coefficients
in $\F_p^{2g+1}$ and with nonintersecting  supports.

\subsection{Distinguished solution over $\C$ and solutions $J^m(z)$ over $\F_p$} Let us compare the distinguished
solution $I(z)=(z_2-z_1)^{-1/2-g}L(\la(z))$ in \Ref{IK}, and the expansion \Ref{KK vec}.
For any $\vec m=(m_0,\dots, m_{a+1})\in M$, define
\bean
\label{pol sol}
&& 
J_{\vec m}(z) =
(z_2-z_1)^{(p-1)/2-g + m_{a+1}p^{a+1}+(p+\dots + p^a)(p-1)/2}K_{\vec m}
\Big(\frac{z_3-z_1}{z_2-z_1}, \dots, \frac{z_{2g+1}-z_1}{z_2-z_1}\Big).
\eean

\begin{thm} The following statements hold.
\begin{enumerate}
\item[(i)] For any $\vec m\in M$,  we have  $J_{\vec m}(z)\in \F_p[z]^{2g+1}$.

\item[(ii)] For any $\vec m\in M$,   the polynomial vector
$J_{\vec m}(z)$ is a solution of the KZ equations \Ref{KZ}.

\item[(iii)] The $\F_p[z^p]$-module spanned by $J_{\vec m}(z)$, $\vec m\in M$, coincides with the 
 $\F_p[z^p]$-module $\mc M_{g, p}$ spanned by $I^m(z)$, $m=0,\dots,g-1$.

\end{enumerate}

\end{thm}

\begin{proof}
We have
\bea
\notag
&&
J_{\vec m}(z) 
=
(-1)^{a(p-1)/2} \binom{2m_{a+1}}{m_{a+1}}
\\
\notag
&&
\phantom{aa}
\times \prod_{j=1}^a 
(z_2-z_1)^{((p-1)/2-m_j +m_{j+1}p)p^j} C^{m_{j+1}}_{m_j}
\Big(\big(\frac{z_3-z_1}{z_2-z_1}\big)^{p^j}, \dots, \big(\frac{z_{2g+1}-z_1}{z_2-z_1}\big)^{p^j}\Big)
\\
\notag
&&
\phantom{aa}
\times
(z_2-z_1)^{(p-1)/2-g +m_1p} K^{m_1}\Big(\frac{z_3-z_1}{z_2-z_1}, \dots \frac{z_{2g+1}-z_1}{z_2-z_1}\Big),
\eea
where
\bea
(z_2-z_1)^{(p-1)/2-g +m_1p} K^{m_1}\Big(\frac{z_3-z_1}{z_2-z_1}, \dots \frac{z_{2g+1}-z_1}{z_2-z_1}\Big)
= J^{m_1}(z)
\eea
is a solution of the KZ equations \Ref{KZ}, see \Ref{CFoo}, and  each factor
\bea
(z_2-z_1)^{((p-1)/2-m_j +m_{j+1}p)p^j} C^{m_{j+1}}_{m_j}
\Big(\big(\frac{z_3-z_1}{z_2-z_1}\big)^{p^j}, \dots, \big(\frac{z_{2g+1}-z_1}{z_2-z_1}\big)^{p^j}\Big)
\eea
is a polynomial in $\F_p[z^p]$.  This proves parts (i-ii) of the theorem. Part (iii) follows from the identity
\bea
K_{\vec m=(g,m_1)}(z) = \binom{2m_1}{m_1} J^{m_1}(z).
\eea
\end{proof}

\end{document}